\documentclass[11pt,reqno]{amsart}
\usepackage[top=1in, bottom=1in, left=1in, right=1in]{geometry}
\usepackage{times, amsthm, amssymb, amsmath, amsfonts, amsthm, bm, graphicx, mathrsfs,mathtools}
\usepackage[all]{xy}
\usepackage[usenames,dvipsnames]{color}
\usepackage[colorlinks=true,linkcolor=blue,urlcolor=blue, citecolor=blue]%
  {hyperref}
\usepackage[alphabetic]{amsrefs}
\usepackage{array,enumitem,multicol}

\usepackage{pgf, tikz}
\usetikzlibrary{arrows, automata}
\usetikzlibrary{patterns,calc,arrows,shapes,positioning}

\usepackage{url}

\DefineSimpleKey{bib}{myurl}

\newcommand\myurl[1]{\url{#1}}

\BibSpec{webpage}{%
  +{}{\PrintAuthors} {author}
  +{,}{ \textit} {title}
  +{}{ \parenthesize} {date}
  +{,}{ \myurl} {myurl}
  +{,}{ } {note}
  +{.}{ } {transition}
}

\newcommand{\excise}[1]{}%{$\star$\textsc{#1}$\star$}

\newcommand{\rvline}{\hspace*{-\arraycolsep}\vline\hspace*{-\arraycolsep}}

\newtheorem{theorem}{Theorem}[section]

\newtheorem{cor}[theorem]{Corollary}

\newtheorem{prop}[theorem]{Proposition}

\newtheorem*{definition*}{Definition}

\newtheorem{question}[theorem]{Question}

\newtheorem{theorem/definition}[theorem]{Theorem/Definition}

\theoremstyle{definition}
\newtheorem{example}[theorem]{Example}
\newtheorem{remark}[theorem]{Remark}

\def\<{\langle}
\def\>{\rangle}

\numberwithin{equation}{section}
\begin{document}%%%%%%%%%%%%%%%%%%%%%%%%%%%%%
%%%%%%%%%%%%%%%%%%%%%%%%%%%%%%%%%%%%%%

\vspace*{-12mm}
\mbox{}
\title[Knockout: a Markov chain anaylsis]{First is the worst, second is the best? \\ a Markov chain analysis of the basketball game knockout}

\author{Andrew Flatz}
\address{Mathematics Department \\
University of Wisconsin - River Falls.}
\email{andrew.flatz@my.uwrf.edu}

\author{Michael C. Loper}
\address{Mathematics Department \\
University of Wisconsin - River Falls.}
\email{michael.loper@uwrf.edu}

\author{Lezlie Weyer}
\address{Mathematics Department \\
University of Wisconsin - River Falls.}
\email{lezlie.weyer@my.uwrf.edu}

%\subjclass[2010]{}
\begin{abstract} 
The game of Knockout is a classic playground game played with two basketballs. This paper uses a Markov process to analyze each player's probability of winning the game given their starting position in line and shooting percentages, assuming all players are equally skilled. The two-player case is solved in general for any probability of a long shot and short shot shooting percentage and the $n$-player case with $n > 2$ is solved numerically. In doing so, this paper answers the question of whether or not the playground wisdom of ``first is the worst, second is best'' is true. We also examine the average number of rounds it takes before the game ends, analyze trends in the data to recommend tips to win at Knockout, and provide questions in the case of players not being equally skilled.
\end{abstract}
\maketitle

\mbox{}
\vspace*{-16mm}
%\setcounter{tocdepth}{1}
%\tableofcontents

%%%%%%%%%%%%%%%%%%%%%%%%%%%%%%%%%%%%%%
\setcounter{section}{1}
\section*{Introduction}
%%%%%%%%%%%%%%%%%%%%%%%%%%%%%%%%%%%%%%
\label{sec:intro}
The phrase ``First is the worst, second is the best'' is often heard as a common playground refrain. But is this phrase more than just a child's saying? Could it be true there is actual wisdom behind it? Knockout, sometimes called Lightning, Elimination, Gotcha, or Bumpout, is a standard schoolyard game played with two basketballs in which players race to make baskets as quickly as they can in the hopes of eliminating the player in front of them \cite{Rosenthal}. At the start of a game of Knockout, it is not unusual for a scuffle to result as many players desire the second position in line. At the same time arguments happen about who must start in the most undesirable position, first, because the first player in line is often also the first person eliminated.

The aim of this article is to discover whether or not a player's starting position actually makes a difference in the game of Knockout. It turns out, that under reasonable assumptions, the starting position does affect the probability of each player winning the game. Further, the playground wisdom above is true at least half of the time:  first truly is the worst, and while not exactly true that second is the best, being in the second position does result in a considerable increase in winning probability.

We model Knockout stochastically with a Markov process: that is, we describe possible states of the game corresponding to which players are shooting and which players are in danger of being eliminated and analyze the probabilities of moving between these states. Markov chains are a common technique used in probability and linear algebra to model anything from a drunkard's walk to fluctuations in the stock market and weather forecasting. Focusing on applications in sports in particular, Markov processes have been useful in predictive models of baseball, hockey, and football \cite{bukiet, MR2360905, MR2011961, MR3409642, Kolbush, Wang}. In the sport of basketball, Paul Kvam and Joel Sokol combine logistic regression with Markov chain methods to predict the outcome of NCAA tournament games \cite{Kvam}. Jian Shi and Kai Song use Markov methods in order to predict outcomes of NBA basketball games \cite{MR4336365}. However, the authors could find no literature exploring Knockout.

\subsection*{Outline}
Section~\ref{sec:prelim} contains a brief background on Markov chains, the rules of the game of Knockout, and a description of the model used throughout this paper. The case of the two-player game of Knockout is analyzed in Section~\ref{sec:twoplayers} with the general case of the $n$-player game (for $n>2$) studied in Section~\ref{sec:nplayers}. The expected length of a game is examined in Section~\ref{sec:expectedlength}. In Section~\ref{sec:discussion}, we switch to a numerical approach to explore general trends and offer tips to improve one's chances at winning a game of Knockout. Finally, Section~\ref{sec:unequal} briefly discusses what changes if one removes the assumption that players are equally skilled and asks questions for future work.

%%%%%%%%%%%%%%%%%%%%%%%%%%%%%%%%%%%%%%
\subsection*{Acknowledgements}
%%%%%%%%%%%%%%%%%%%%%%%%%%%%%%%%%%%%%%
\label{subsec:acknowledgements}

The authors thank Daniel Swenson for a helpful conversation in the beginning stages of this project and Ben Strasser for providing feedback that improved the readability of this paper.

%%%%%%%%%%%%%%%%%%%%%%%%%%%%%%%%%%%%%%
\section{Background}
%%%%%%%%%%%%%%%%%%%%%%%%%%%%%%%%%%%%%%
\label{sec:prelim}
In this section, we first review some basic theory about Markov chains, then describe the rules of the game of Knockout and our stochastic model of the game.

\subsection{Background on absorbing Markov chains}
A more in-depth explanation of absorbing Markov chains can be found in \cite{grinstead}*{Ch 11}. We begin with a few definitions:
A \textit{Markov process} or \textit{Markov chain} is a state-based stochastic process that specifies a sequence of possible events where each future event is independent of past events and depends only on the current state. An \textit{absorbing state} of a Markov chain is a state, that once entered, cannot be left. A \textit{transient state} is a state that is not absorbing. An \text{absorbing Markov chain} is a Markov chain where every state can reach an absorbing state. Letting $T$ be the transition matrix of an absorbing Markov chain with $t$ transient states and $r$ absorbing states, $T$ is often written in \textit{canonical form} where the rows of $T$ denote sources and the columns of $T$ are represented by destinations. Therefore the probability of moving from state $i$ to state $j$ is $T_{ij}$. In canonical form, the states are ordered so that the absorbing states are last. In this way $T$ can be written as a block matrix
$$T = \begin{bmatrix} Q & R \\ \mathbf{0} & I_r \end{bmatrix}$$
where $Q$ is a $t$-by-$t$ matrix that respresents the probabilities of going from each transient state to every other transient state, $R$ is a $t$-by-$r$ matrix that represents the probability of moving from each transient state to each absorbing state, $\mathbf{0}$ is the $r$-by-$t$ zero matrix, and $I_r$ is the $r$-by-$r$ identity matrix. The matrix $N:=(I-Q)^{-1}$ is often called the \textit{fundamental matrix} of an absorbing Markov chain.

The following theorem about absorbing Markov chains is well-known, but a proof is included here to keep this article as self-contained as possible. This theorem will do the heavy lifting in discovering the optimal starting position.

\begin{theorem}\label{thm:markovsoln}
	In an absorbing Markov chain, the probability of eventually being absorbed in the absorbing state $j$ when starting in transient state $i$ is given by the $(i,j)$-entry of the matrix $NR.$
\end{theorem}

\begin{proof}
The entry $R_{i,j}$ is the probability of being absorbed into absorbing state $j$ from transient state $i$ after one time step and similarly, the entry $(Q^nR)_{i,j}$ is probability of being absorbed into absorbing state $j$ from transient state $i$ after $n-1$ time steps. The probability of eventually being absorbed in the absorbing state $j$ when starting in transient state $i$ is the sum over all time steps. Hence, the probability of of eventually being absorbed in the absorbing state $j$ when starting in transient state $i$ is the $(i,j)$-entry of
$$(I + Q + Q^2 + \cdots)R = (I-Q)^{-1}R = NR,$$
where the first equality follows from the sum of a geometric series.
\end{proof}

\subsection{Background on Knockout}
\label{subsec:rules}
A game of Knockout begins with a line of $n$ players standing at the free throw line\footnote{Some variants of Knockout have players standing elsewhere like the three-point line or even half court, but for simplicity we'll assume here that players start at the free throw line. Starting elsewhere on the court will not change the analysis.}. The first two players have basketballs. The first player takes a shot from the free throw line, followed by the second player shooting from the free throw line. If a player misses their shot from the free throw line, they rebound their miss and may take their next shot from anywhere on the court. If the second player makes a basket before the first player in line does, the first player in line is eliminated. Both players rebound their ball and pass it to the next two players in line. The first player is eliminated and the second player takes their place at the end of the line. If the first player makes a basket before the second player, the first player rebounds their make and passes their ball to the third player in line, while the second player continues to shoot. The first player remains in the game and cycles to the end of the line. The third player may eliminate the second player by making a basket before them. The game continues in this manner with each player attempting to eliminate the player in front of them by making a basket before them, while at the same time trying to avoid getting eliminated by the player behind them. Eventually every player except for one is eliminated. This last player remaining is crowned the champion.

\subsection{Description of Model}
In order to better model the game we make two assumptions:
\begin{enumerate}
	\item players take turns shooting
	\item all players are equally skilled
\end{enumerate}
%While speed is an important part of the game of Knockout, it is very difficult to model the speed and endurance of each player in line. In the event that a player misses a long shot, it is also very difficult to model where the rebound will end up. A longer rebound puts the player further from the rim and it would take longer for that player to move to a position for them to take a short shot with the highest probability of making a basket. Sometimes players may also decide it is in their best interest to shoot a longer second or third shot if the player about to eliminate them is about to shoot. These considerations, while an important part of the game, are difficult to model and in the authors' opinions will most likely cancel out if all players are equally skilled. Therefore, to make the model easier to work with, analyze, and understand we have made the decision that players will take turns shooting. 

It may be helpful to think of having two different color basketballs: one red and one blue. The red ball will be shot first with the blue ball being shot immediately after. After the blue ball is shot, the red ball will be shot again, and so on. As a result of Assumption (1), it is in every players' best interest to take a shot with the highest probability of going in if they miss their long shot (LS) from the free throw line. This shot is most likely a lay-up and we call any shot that is not a long shot a ``short shot'' (SS). It is usual that the probability of making a short shot is greater than a long shot.

In order to model the game of Knockout, we first break up a game played with $n$ players into a series of $n-1$ rounds. Each round culminates with a player being eliminated. In a game with $n$ players, Round 1 starts with $n$ players and ends with $n-1$ players. In general, Round $k$ (for $k \le n-1$) begins with $n-k+1$ players and ends with $n-k$ players. After $n-1$ rounds, there have been $n-1$ players eliminated and the lone player remaining is the champion.

Each round can be partitioned into discrete game-states. Each game-state is determined by which two shots will be shot next (e.g. red ball short shot then blue ball long shot) and in the case of a two-player round, which player is in danger of being eliminated. In order to transition between game-states, two shots are usually taken (the only exceptions being cases where the player making a shot with the red ball eliminates someone). Because the probabilities of making a short shot and long shot are assumed to be given, the probability of moving between each pair of game-states can be calculated. Thus each round forms a Markov chain. Even better, because eliminated players can never rejoin the game, each elimination game-state is a steady state where there is a probability of 1 of remaining in that state and a probability of 0 of going to any other state. Once such an absorbing state is reached, the game starts another round with one fewer player. In this way, our model of an $n$-player game of Knockout consists of $n-1$ rounds of absorbing Markov chains.

The second assumption that all players are equally skilled enable reduces the game to three parameters: $p$, the probability that each player makes a long shot from the free-throw line, $q$, the probability that each player makes a short shot, and $n$, the number of players in line at the start of the game. 

The next theorem calculates the expected number of rounds needed to end a game of Knockout. This theorem is well-known for general absorbing Markov chains, but we adapt it here in the context of our model of Knockout.

\begin{theorem}\label{thm:expected}
	The expected number of steps before someone is eliminated in a round of Knockout is given by the first entry in the column matrix
$$\vec{t} = N\vec{1}$$
where $\vec{1}$ is the column vector consisting of all $1$s.
\end{theorem}

\begin{proof}
	The $(i,j)$th entry of $Q$ is given by the probability that transient game-state $j$ is reached after one step after starting in transient game-state $i$. Similarly $Q^k$ is the matrix whose $(i,j)$th entry gives the probability that game-state $j$ is reached after $k$ steps after starting in game-state $i$. Since our game always starts in the 1st game-state, we will only care about the first row of $Q^k$. The sum of the first row of $Q^k$ is equal to the first entry of $Q^k \vec{1}$ and gives the probability that the game is in any transient game-state after $k$ steps. Because the absorbing game-states are exactly the states where a player is eliminated, this probability is the same as the probability that no one is eliminated after $k$ steps. Let $p_k$ be the probability that no one is eliminated after $k$ steps.

Notice that the probability that a player is eliminated after exactly $k$ steps is equal to $p_{k} - p_{k+1}$. Computing the expected number of rounds gives
\begin{align*}
	1*P(&\text{someone is eliminated after one step})  \\ &+ 2*P(\text{someone is eliminated after two steps})& \\ &+ 3*P(\text{someone is eliminated after three steps}) + \cdots \\
&= \sum_{k=0}^\infty k*P(\text{someone is eliminated after } k \text{ steps}) \\
	&= \sum_{k=0}^\infty k*(p_{k} - p_{k+1}) \\
	&= 1(p_1 - p_2) + 2(p_2 - p_3) + 3(p_3 - p_4) + 4(p_4-p_5) + \cdots \\
	&= p_1 - p_2 + 2p_2 - 2p_3 + 3p_3 - 3p_4 + 4p_4 - 4p_5 + \cdots \\
	&= p_1 +  (- p_2 + 2p_2) + (- 2p_3 + 3p_3) + (- 3p_4 + 4p_4) - 4p_5 + \cdots \\
	&= p_1 + p_2 + p_3 + p_4  + \cdots \\
	&= \sum_{k=0}^\infty p_k \\
	&= \sum_{k=0}^\infty Q^k \vec{1} \\
	&= \left( \sum_{k=0}^\infty Q^k \right) \vec{1} \\
	&= N\vec{1}
\end{align*}
\end{proof}

%%%%%%%%%%%%%%%%%%%%%%%%%%%%%%%%%%%%%%
\section{The Two Player Game}
%%%%%%%%%%%%%%%%%%%%%%%%%%%%%%%%%%%%%%
\label{sec:twoplayers}

Every game of Knockout eventually leads to a final round where two players go head to head in a flurry of running and shots with each player's goal being to get in position behind the other player and eliminate them. In practice this tends to be the most exhausting part of the game, however in this model we assume players are equally skilled and ignore matters of player fatigue. Since every game of Knockout eventually reduces to this final round, it makes sense to examine this round first. Notice the two-player game is identical to the final round of an $n$-player game.

The two-player game has seven game-states. Since there are only two players, the balls never change hands. Player 1 ($P_1$) will always shoot the red ball, and Player 2 ($P_2$) will always shoot the blue ball. This means $P_1$ shoots first in each of the game-states, which are enumerated below.
\begin{enumerate}
	\item [(G1)] $P_2$ can win: $P_1$ shooting a long shot, followed by $P_2$ shooting a long shot
	\item [(G2)]$P_2$ can win: $P_1$ shooting a short shot, followed by $P_2$ shooting a short shot
	\item [(G3)]$P_2$ can win: $P_1$ shooting a short shot, followed by $P_2$ shooting a long shot
	\item [(G4)]$P_1$ can win: $P_1$ shooting a long shot, followed by $P_2$ shooting a short shot
	\item [(G5)]$P_1$ can win: $P_1$ shooting a short shot, followed by $P_2$ shooting a short shot
	\item [(G6)]$P_2$ is eliminated which means $P_1$ is the winner
	\item [(G7)]$P_1$ is eliminated which means $P_2$ is the winner
\end{enumerate}
Figure~\ref{fig:directedgraph} is a weighted directed graph with the game-states as vertices and the weight of each edge as the probability of moving from one game-state to another. The absorbing game-states (G6) and (G7) are represented using rectangles instead of circles.

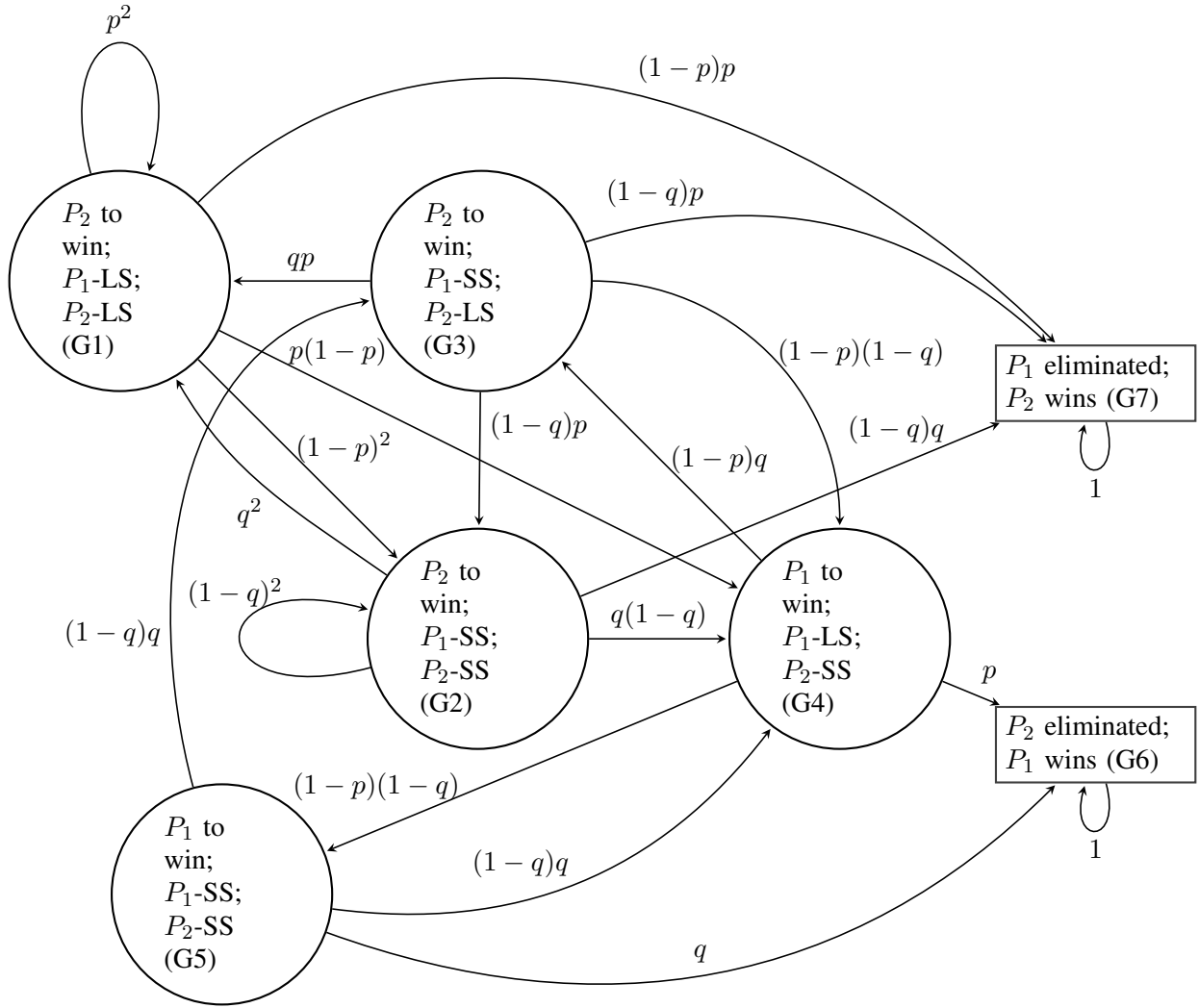
\begin{figure}
\begin{center}
    \begin{tikzpicture}[
            > = stealth, % arrow head style
            shorten > = 1pt, % don't touch arrow head to node
            auto,
            node distance = 5cm, % distance between nodes
            semithick % line style
        ]

        \tikzstyle{state}=[circle,
            draw = black,
            thick,
            fill = white,
            minimum size = 4mm
]

	\tikzstyle{abstate}=[rectangle,thick,draw=black!75,
  			  fill=white,minimum size=4mm, text width=2.5cm]

        \node[state, text width=1.6cm] (G1) {$P_2$ to win; \\  $P_1$-LS; $P_2$-LS (G1)};
        \node[state, text width=1.6cm] (G2) [below right of =G1, node distance=7cm] {$P_2$ to win; \\ $P_1$-SS; $P_2$-SS (G2)};
        \node[state, text width=1.6cm] (G3) [right of=G1] {$P_2$ to win; \\ $P_1$-SS; $P_2$-LS (G3)};
        \node[state, text width=1.6cm] (G4) [below right of=G3, node distance=7cm] {$P_1$ to win; \\ $P_1$-LS; $P_2$-SS (G4)};
        \node[state, text width=1.6cm] (G5) [below left of=G2] {$P_1$ to win; \\ $P_1$-SS; $P_2$-SS (G5)};
        \node[abstate] (G7) [above right of=G4] {$P_1$ eliminated; $P_2$ wins (G7)};
        \node[abstate] (G6) [below of=G7] {$P_2$ eliminated; $P_1$ wins (G6)};

        \path[->] (G1) edge[loop above] node {$p^2$} (G1)
                                edge node [xshift=-5pt, yshift=-5pt] {$(1-p)^2$} (G2)
			edge node [very near start, xshift=-2pt, yshift=-5pt] {$p(1-p)$} (G4)
        			edge [out=45, in=135] node {$(1-p)p$} (G7)
        		   (G2) edge [out=145, in=-60] node [xshift=2pt]  {$q^2$} (G1)
			edge [loop left, anchor=south] node[xshift=-2pt, yshift=8pt] {$(1-q)^2$} (G2)
			edge node {$q(1-q)$} (G4)
			edge node [very near end, xshift=2pt, yshift=-4pt] {$(1-q)q$} (G7)
        		   (G3) edge [anchor=south] node {$qp$} (G1)
			edge [near start] node {$(1-q)p$} (G2)
			edge [anchor=west, out=0, in = 90] node {$(1-p)(1-q)$} (G4)
			edge [bend left, near start] node {$(1-q)p$} (G7)
		  (G4) edge [anchor=west] node {$(1-p)q$} (G3)
			edge [anchor=south, very near end] node [yshift=8pt] {$(1-p)(1-q)$} (G5)
			edge node {$p$} (G6)
		  (G5) edge [out=105, in = 190, near start] node {$(1-q)q$} (G3)
			edge [bend right] node [yshift=-2pt]{$(1-q)q$} (G4)
			edge [out=-20, in=-135] node {$q$} (G6)
		  (G6) edge [loop below] node {$1$} (G6)
		  (G7) edge [loop below] node {$1$} (G7);
    \end{tikzpicture}
\end{center}
\caption{Directed graph representing a 2-person game of Knockout}
\label{fig:directedgraph}
\end{figure}

This graph is encoded in the following transition matrix
\begin{equation}\label{eqn:twoplayer}
T = \begin{bmatrix}
p^2 & (1-p)^2 & 0 & p(1-p) & 0 & 0 & (1-p)p \\
q^2 & (1-q)^2 & 0 & q(1-q) & 0 & 0 & (1-q)q \\
qp & (1-q)p & 0 & q(1-p) & 0 & 0 & (1-q)p \\
0 & 0 & (1-p)q & 0 & (1-p)(1-q) & p & 0 \\
0 & 0 & (1-q)q & 0 & (1-q)(1-q) & q & 0 \\
0 & 0 & 0 & 0 & 0 & 1 & 0 \\
0 & 0 & 0 & 0 & 0 & 0 & 1
\end{bmatrix}
\end{equation}

\begin{theorem}
\label{thm:2player}
	In a round of Knockout with two players, $P_1$ followed by $P_2$, where both players have a long shot percentage of $p$ and a short shot percentage of $q$, the probability that $P_2$ is eliminated is $\dfrac{1}{3-p}$.
\end{theorem}

\begin{proof}
	The result follows from an application of Theorem~\ref{thm:markovsoln} to the transition matrix $T$ in Equation~\ref{eqn:twoplayer}.
\end{proof}

\begin{cor}
	In a game of Knockout with two players, the probability of $P_1$ winning lies in the interval $\left[\frac{1}{3},\frac{1}{2}\right)$.
\end{cor}

\begin{proof}
\label{cor:2player}
	A game of Knockout with two players consists of a single round with two players. By Theorem~\ref{thm:2player}, the probability that $P_1$ wins $\dfrac{1}{3-p}$. If $p=1$, then the game would never end as both players would never miss their long shot. So $p \in [0,1)$, which forces the probability that $P_1$ wins to lie in the desired interval.
\end{proof}

\begin{remark}
	It is surprising to note that in the two player game, the probability of making a short shot does not affect each player's win probability.
\end{remark}

%%%%%%%%%%%%%%%%%%%%%%%%%%%%%%%%%%%%%%
\section{More than two players}
%%%%%%%%%%%%%%%%%%%%%%%%%%%%%%%%%%%%%%
\label{sec:nplayers}
Although every game of Knockout will eventually lead to a showdown between two players, most games begin with many more than two players. In order to fully analyze which starting position is best, we must investigate the $n$-player round where $n>2$ and see how to transition from the $n$-player round to the $(n-1)$-player round. The good news is that as $n$ increases, the number of game-states in a round grows in a linear fashion.

\begin{prop}
\label{prop:gamestates}
	A round of Knockout with $n>2$ players has $6n$ game-states.
\end{prop}

\begin{proof}
	Since there are two balls and two different types of shots there are 4 combinations of shot selections: a long shot followed by a long shot, a long shot followed by a short shot, a short shot followed by a long shot, and a short shot followed by a short shot. Two players must be concurrent in line in order to be shooting at the same time (with the understanding that the last player and first player are also considered concurrent in line. Enumerate the players $P_1, P_2,...,P_n$. Since the players must be concurrent in line, they can either be in the shooting order of $P_k$ followed by $P_{k+1}$ or $P_{k+1}$ followed by $P_k$. Therefore, there are $4*2n = 8n$ transient game-states. Adding the $n$ absorbing states (one for eliminating each player) gives $9n$ game-states. However, exactly $3n$ of these game-states are impossible:
\begin{itemize}
	\item $P_{k+1}$ cannot shoot a long shot followed by $P_{k}$ shooting a long shot
	\item $P_{k+1}$ cannot shoot a short shot followed by $P_k$ shooting a long shot
	\item $P_k$ cannot shoot a long shot followed by $P_{k+1}$ shooting a short shot
\end{itemize}
 These impossible game-states all stem from the fact that $P_k$ must shoot a long shot before $P_{k+1}$ shoots a long shot. Indeed the first two families of impossible states follow immediately from this fact. The last family of impossible states cannot occur because in order for $P_{k+1}$ to be shooting a short shot, $P_{k+1}$ must have already shot a long shot. This gives exactly $6n$ game-states.
\end{proof}

Notice when $n=2$, there are 7 game-states as opposed to $12$ game-states. The reason for this is explained in Remark \ref{rmk:gamestates} below.

Since there are $6n$ game-states, each transition matrix has size $6n \times 6n$ so while game-states grow linearly, the transition matrix grows quadratically as $n$ increases. The good news is that ordering the game-states according to the type of shots that players take results in the transition matrix having a predictable pattern. It can be written as a $6 \times 6$ block matrix using only the variables $p$ and $q$ and the four matrices $I$, the $n \times n$ identity matrix, $0$, the $n \times n$ zero matrix, $A = A_n$, an $n \times n$ cyclic permutation matrix below, and $B = B_n = A^2$, another cyclic permutation matrix. Below are examples of $A$ and $B$ in the case of $n=3$ and $n=4$.
$$A_3 = \begin{bmatrix} 0 & 1 & 0 \\ 0 & 0 & 1 \\ 1 & 0 & 0 \end{bmatrix}, \hspace{1cm} B_3 = \begin{bmatrix} 0 & 0 & 1 \\ 1 & 0 & 0 \\ 0 & 1  & 0 \end{bmatrix}$$
\vspace{0.5cm}
$$A_4 = \begin{bmatrix} 0 & 1 & 0 & 0 \\ 0 & 0 & 1 & 0 \\ 0 & 0 & 0 & 1 \\ 1 & 0 & 0 & 0 \end{bmatrix}, \hspace{1cm} B_4 = \begin{bmatrix} 0 & 0 & 1 & 0 \\ 0 & 0 & 0 & 1 \\ 1 & 0 & 0 & 0 \\ 0 & 1 & 0 & 0 \end{bmatrix}, \hspace{1cm} I_4 = \begin{bmatrix} 1 & 0 & 0 & 0 \\ 0 & 1 & 0 & 0 \\ 0 & 0 & 1 & 0 \\ 0 & 0 & 0 & 1 \end{bmatrix}$$

With this notation, the transition matrix for the $n$ player case is the following block diagonal matrix (where the lines indicate how the matrix is in cannonical form with $Q$, $R$, $0$, and $I$.
\begin{equation}\label{transitionmatrix}
T =\begin{bmatrix}
p^2 B & (1-p)^2I  & 0 & 0 & p(1-p)A & \rvline & (1-p)pI \\
 & & & & & \rvline & \\
q^2B & (1-q)^2I & 0 & 0 & q(1-q)A & \rvline & (1-q)qI \\
 & & & & & \rvline & \\
0 & 0 & (1-q)^2I & (1-q)qA & 0 & \rvline & qI \\
 & & & & & \rvline & \\
qpB & (1-q)(1-p)I & 0 & 0 & q(1-p)A & \rvline & (1-q)pA \\
 & & & & & \rvline & \\
0 & 0 & (1-p)(1-q)I & (1-p)qA & 0 & \rvline & pI \\
\hline
 & & & & & \rvline & \\
0 & 0 & 0 & 0 & 0 & \rvline & I
\end{bmatrix}
\end{equation}

The game-states are as follows, broken up into the blocks of $n$ that correspond to the block matrices in the transition matrix in Equation~\ref{transitionmatrix}:

\begin{multicols}{2}
{\bf \underline {Block 1}}
\begin{enumerate}
	\item [$(1)$]$P_1$ takes LS, then $P_2$ takes LS
	\item [$(2)$]$P_2$ takes  LS, then $P_3$ takes LS \\
\vdots

	\item [{$(n-1)$}] $P_{n-1}$ takes LS, then $P_n$ takes LS
	\item [$(n)$]$P_{n}$ takes LS, then $P_1$ takes LS
\end{enumerate}

{\bf \underline {Block 2}}
\begin{enumerate}
	\item[$(n+1)$] $P_1$ takes SS, then $P_2$ takes SS
	\item [$(n+2)$]$P_2$ takes SS, then $P_3$ takes SS \\
\vdots

	\item [{$(2n-1)$}] $P_{n-1}$ takes SS, then $P_n$ takes SS
	\item [$(2n)$]$P_{n}$ takes SS, then $P_1$ takes SS
\end{enumerate}

{\bf \underline {Block 3}}
\begin{enumerate}
	\item[$(2n+1)$] $P_2$ takes SS, then $P_1$ takes SS
	\item [$(2n+2)$]$P_3$ takes SS, then $P_2$ takes SS \\
\vdots

	\item [{$(3n-1)$}] $P_{n}$ takes SS, then $P_{n-1}$ takes SS
	\item [$(3n)$]$P_{1}$ takes SS, then $P_{n}$ takes SS
\end{enumerate}

\columnbreak

{\bf \underline {Block 4}}
\begin{enumerate}
	\item[$(3n+1)$] $P_1$ takes SS, then $P_2$ takes  LS
	\item [$(3n+2)$]$P_2$ takes SS, then $P_3$ takes  LS \\
\vdots

	\item [{$(4n-1)$}] $P_{n-1}$ takes SS, then $P_n$ takes LS
	\item [$(4n)$]$P_{n}$ takes SS, then $P_1$ takes LS
\end{enumerate}

{\bf \underline {Block 5}}
\begin{enumerate}
	\item[$(4n+1)$] $P_2$ takes LS, then $P_1$ takes SS
	\item [$(4n+2)$]$P_3$ takes LS, then $P_2$ takes SS \\
\vdots

	\item [{$(5n-1)$}] $P_{n}$ takes LS, then $P_{n-1}$ takes SS
	\item [$(5n)$]$P_{1}$ takes LS, then $P_{n}$ takes SS
\end{enumerate}

{\bf \underline {Block 6 - Absorbing States}}
\begin{enumerate}
	\item[$(5n+1)$] $P_1$ is eliminated
	\item [$(5n+2)$]$P_2$ is eliminated \\
\vdots

	\item [$(6n)$]$P_{n}$ is eliminated
\end{enumerate}
\end{multicols}

\begin{remark}
\label{rmk:gamestates}
	The assumption that $n>2$ in Proposition \ref{prop:gamestates} is crucial. As one sees from Figure \ref{fig:directedgraph}, the 2-player game has only seven game-states as opposed to 12. The reason for the reduction of 5 game-states is due to the fact that with only two players, $P_1$ will always shoot before $P_2$. This eliminates one game-state from each of of the 5 transient states blocks (Blocks 1-5 above). In the game-state labeling of Figure \ref{fig:directedgraph}, (G1) corresponds to the lone game-state in Block 1, (G2) corresponds to Block 2, (G3) corresponds to Block 4, (G4) corresponds to Block 5, and (G5) corresponds to Block 3.
\end{remark}

Now that the transition matrix for an $n$-player round is known, notice that when a player is eliminated from a round with $n$ players, the game changes to an $(n-1)$-player round. Further, the $(n-1)$ player round begins with both players shooting a long shot. So once the probabilities of each player being eliminated from an $n$-player round is known, recursion can be used to find the probability that of each player winning an $n$-player game. The only difficulty is in keeping track of the positions of each player in the new $(n-1)$-player round when $P_i$ is eliminated from the $n$-player round. Notice the game starts with $P_{i+1}$ taking the first (long) shot and $P_{i+2}$ taking the second (long) shot. This is recorded in the following proposition whose proof is omitted because it is a straightforward matter of careful bookkeeping.

\begin{prop}
Let $E_i$ be the probability that $P_i$ is eliminated in an $n$-player round and let $W_j$ be the probability that $P_j$ wins an $(n-1)$-player game. Then the probability that $P_k$ wins an $n$-player game is
$$\sum_{i=1}^{n-1} E_{[j+i-1]+1}W_{[n-2-i] +1}$$
where $[j+i-1]$ is the unique representation of $j+i-1$ mod $n$ that lies between 0 and $n-1$ and $[n-2-i]$ is the unique representative of $n-2-i$ mod $(n-1)$ that lies between $0$ and $n-2$.
\end{prop}

\setcounter{MaxMatrixCols}{18}
\begin{example}\label{example:threeplayer}
	Consider the case of the three player game, played with players $A$, $B$, and $C$, starting in that order. Then using the transition matrix from Equation \ref{transitionmatrix} becomes \newline
 \resizebox{1\textwidth}{!}{
$\begin{bmatrix}
0 & 0 & p^2 & 	(1-p)^2 & 0 & 0 & 	0 & 0 & 0 & 		0 & 0  & 0 & 		0 & p(1-p) & 		0 & (1-p)p & 0 & 0 \\
p^2 & 0 & 0 & 	0 & (1-p)^2 & 0 & 	0 & 0 & 0 & 		0 & 0 & 0 & 		0 & 0 & p(1-p) & 	0 & (1-p)p & 0 \\
0 & p^2 & 0 & 	0 & 0 & (1-p)^2 & 	0 & 0 & 0 &		0 & 0 & 0 & 		p(1-p) & 0 & 0 &	0 & 0 & (1-p)p \\

0 & 0 & q^2 & 	(1-q)^2 & 0 & 0 &	0 & 0 & 0 & 		0 & 0 & 0 &		0 & q(1-q) & 0 & 	(1-q)q & 0 & 0 \\
q^2 & 0 & 0 & 	0 & (1-q)^2 & 0 & 	0 & 0 & 0 &		0 & 0 & 0 & 		0 & 0 & q(1-q) & 	0 & (1-q)q & 0 \\
0 & q^2 & 0 & 	0 & 0 & (1-q)^2 & 	0 & 0 & 0 & 		0 & 0 & 0 & 		q(1-q) & 0 & 0 & 	0 & 0 & (1-q)q \\

0 & 0 & 0 & 		0 & 0 & 0 & 		(1-q)^2 & 0 & 0 &	0 & (1-q)q & 0 &		0 & 0 & 0 & 		q & 0 & 0 \\
0 & 0 & 0 &		0 & 0 & 0 &		0 & (1-q)^2 & 0 & 	0 & 0 & (1-q)q & 		0 & 0 & 0 & 		0 & q & 0 \\
0 & 0 & 0 & 		0 & 0 & 0 &	 	0 & 0 & (1-q)^2 & 	(1-q)q & 0 & 0 & 		0 & 0 & 0 &		0 & 0 & q \\

0 & 0 & qp & 		(1-q)(1-p) & 0 & 0 & 0 & 0 & 0 & 		0 & 0 & 0 & 		0 & q(1-p) & 0 &	0 & (1-q)p & 0 \\
qp & 0 & 0 & 		0 & (1-q)(1-p) & 0 & 0 & 0 & 0 & 		0 & 0 & 0 &		0 & 0 & q(1-p) &	0 & 0 & (1-q)p \\
0 & qp & 0 & 		0 & 0 & (1-q)(1-p) & 0 & 0 & 0 & 		0 & 0 & 0 & 		q(1-p) & 0 & 0 & 	(1-q)p & 0 & 0 \\

0 & 0 & 0 & 		0 & 0 & 0 &		(1-p)(1-q) & 0 & 0 & 0 & (1-p)q & 0 & 	0 & 0 & 0 &		p & 0 & 0 \\
0 & 0 & 0 & 		0 & 0 & 0 &		0 & (1-p)(1-q) & 0 & 0 & 0 & (1-p)q & 	0 & 0 & 0 & 		0 & p & 0 \\
0 & 0 & 0 & 		0 & 0 & 0 &		0 & 0 & (1-p)(1-q) & (1-p)q & 0 & 0 & 	0 & 0 & 0 & 		0 & 0 & p \\

0 & 0 & 0 & 		0 & 0 & 0 &		0 & 0 & 0 &		0 & 0 & 0 &		0 & 0 & 0 &		1 & 0 & 0 \\
0 & 0 & 0 &		0 & 0 & 0 &		0 & 0 & 0 &		0 & 0 & 0 &		0 & 0 & 0 &		0 & 1 & 0 \\
0 & 0 & 0 &		0 & 0 & 0 &		0 & 0 & 0 &		0 & 0 & 0 &		0 & 0 & 0 &		0 & 0 & 1
\end{bmatrix}.$} \newline
Using Theorem \ref{thm:markovsoln}, the probability of Player $A$ being eliminated in the first round is
\begin{center}
\resizebox{1\textwidth}{!}{$\dfrac{\left(p^{3}-2\,p^{2}q+p\,q^{2}+p\,q-2\,q^{2}-p+5\,q-4\right)\left(p^{4}-2\,p^{3}q+p^{2}q^{2}-p^{3}+p^{2}q+4\,p^{2}-3\,p\,q-3\,p+3\right)\left(-1\right)}{\left(p^{4}-2\,p^{3}q+p^{2}q^{2}-3\,p^{3}+5\,p^{2}q-2\,p\,q^{2}+4\,p^{2}-7\,p\,q+4\,q^{2}+3\,p-8\,q+7\right)\left(p^{4}-2\,p^{3}q+p^{2}q^{2}-p^{3}+p^{2}q+4\,p^{2}-3\,p\,q-3\,p+3\right)}$}.
\end{center}
Notice that in contrast to the 2-player case, the probability of making a short shot, $q$, matters here. This expression is also much more complicated than in the two player case. For this reason it does not seem like knowing the general solution for larger games will lead to many insights. In Section \ref{sec:numeric}, we switch to a numerical approach and analyze the data to identify trends.
\end{example}

%%%%%%%%%%%%%%%%%%%%%%%%%%%%%%%%%%%%%%
\section{Expected Length of a Game}
%%%%%%%%%%%%%%%%%%%%%%%%%%%%%%%%%%%%%%
\label{sec:expectedlength}
In this section we compute the expected number of steps before the game ends. By steps, we mean traversing one edge in the directed graph. So a single step usually consists of two shots -- one shot of each ball. Notice not every single step consists of exactly two shots. A player shooting the blue ball can sometimes be eliminated if the player behind them makes their shot with the red ball first. Then the player with the blue ball is eliminated and does not have an opportunity to shoot the blue ball in that round. This can be seen in Figure~\ref{fig:directedgraph} when moving from either of the game-states (G4) or (G5) to game-state (G6).

Since our game is modeled by an absorbing Markov chain, we use the following result which is well-known for general absorbing Markov chains, but which we have adapted to Knockout.

\begin{theorem}\label{thm:expectedlength}
	The expected number of steps before someone is eliminated in a round of Knockout with $2$ or more players is
$$\frac{\left(p+q-3\right)\left(p^{2}-p\,q-2\,p+2\,q+1\right)\left(-1\right)}{\left(q\right)\left(p-3\right)\left(p-1\right)\left(p-q+1\right)}$$
%$$\frac{\left(p\,q-q-1\right)\left(p^{3}q-p\,q^{3}-4\,p^{2}q+3\,p\,q^{2}+q^{3}+p^{2}+3\,p\,q-4\,q^{2}-2\,p+2\,q+1\right)\left(-1\right)}{\left(q\right)\left(p-1\right)\left(p-q+1\right)\left(p^{2}q^{2}-3\,p\,q^{2}+2\,q^{2}-2\right)}$$
where $p$ is the probability of making a long shot and $q$ is the probability of making a short shot. \\
\end{theorem}

\begin{proof}
	Notice this expression is independent of $n$, the number of players. Because all players are equally skilled, there is no difference between a player making a basket, then that same player taking a long shot verses making a basket and passing to another player to take a long shot. Then the theorem follows from Theorem \ref{thm:expected} applied to the matrix from Equation \ref{eqn:twoplayer}. One can check that this result agrees with Theorem \ref{thm:expected} applied to the matrix from Example \ref{example:threeplayer}.
\end{proof}

\begin{cor}
\label{cor:rounds}
	The expected number of steps before someone wins a game of Knockout with $n$ players is
$$\frac{(n-1)\left(p+q-3\right)\left(p^{2}-p\,q-2\,p+2\,q+1\right)\left(-1\right)}{\left(q\right)\left(p-3\right)\left(p-1\right)\left(p-q+1\right)}.$$
\end{cor}

\begin{proof}
	By Theorem \ref{thm:expectedlength}, the expected number of steps until a player is eliminated is independent of the number of players $n$. Since a game with $n$ players requires $n-1$ rounds/eliminations to declare a winner and each elimination is independent, this gives the desired result.
\end{proof}

Taking a closer look at the denominator, and keeping mind both $p$ and $q$ must lie in the interval $[0,1]$, the factor $(p-3)$ is never zero. Additionally, if $q=0$, then no players ever make a short shot. This leads to a never ending game because two players get stuck in an infinite loop of missing short shots. Similarly, if $p=1$, the game never ends because all players constantly make their long shot. Finally, if $q=0$ and $p=1$, then the factor $(p-q+1)$ is zero. This corresponds to the never ending game where no players make their long shot, but all players make their short shot.

\begin{remark}
	The largest game of Knockout, as recorded by Guiness World Records was hosted by the Dallas Mavericks in October, 2015 and involved 701 participants \cite{Guiness}. This game lasted over three hours with the winner winning season tickets to the Mavericks. According to Corollary \ref{cor:rounds}, if all players shot $p=40\%$ and $q=90\%$, this game would require an average 3051 steps or about 6,000 shots before its conclusion.
\end{remark}

%%%%%%%%%%%%%%%%%%%%%%%%%%%%%%%%%%%%%%
\section{Numerical Data and Discussion}\label{sec:numeric}
%%%%%%%%%%%%%%%%%%%%%%%%%%%%%%%%%%%%%%
\label{sec:discussion}
In this section, we investigate the data produced by the model outlined in the sections above and identify trends. This is done with the hope of producing actionable steps that players may take to increase their probability of winning a game of Knockout. We begin by summarizing the trends:
\begin{itemize}
	\item the worst position to start in is first
	\item it is better to be further back in line
	\item Even numbered positions receive a bump in probability of winning, but this bump rapidly decreases the further back in line one starts
	\item The largest difference between a player's best chance of winning and the worst chance of winning occurs when $p \approx 0.5$
\end{itemize}

Figure~\ref{graph:40907} is a graph highlighting all of these trends at once. It is from a game played with 7 players who shoot a short shot percentage of $q = 90\%$ and a long shot percentage of $p = 40\%$. The approximate probabilities of winning are displayed in Table~\ref{table:40907}.

\begin{figure}
\begin{center}
\includegraphics[width=\textwidth]{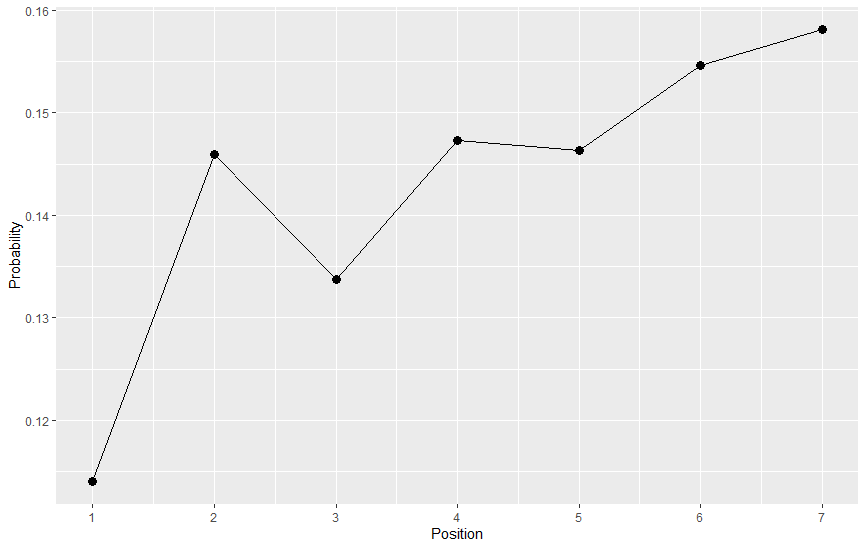}
\end{center}
\caption{Probability of winning by position when $n=7$, $p=0.4$ and $q = 0.9$}
\label{graph:40907}
\end{figure}

\begin{table}
\begin{center}
\begin{tabular}{|c|c|c|c|c|c|c|c|}
\hline
Position: & 1 & 2 & 3 & 4 & 5 & 6 & 7 \\
\hline
Probability of Winning: & 0.11402 & 0.14597 & 0.13370 & 0.14727 & 0.14634 & 0.15462 & 0.15807 \\
\hline
\end{tabular}
\end{center}
\caption{Probability of winning by position when $n=7$, $p=0.4$ and $q = 0.9$}
\label{table:40907}
\end{table}

Notice that $P_1$ has the lowest probability of winning. Then $P_2$ has a much better probability of winning, but $P_3$ has a slightly lower probability of winning compared with the second player. The fourth player has a slightly better probability of winning than $P_2$ and $P_5$ has a slightly worse probability of winning than $P_4$, but still slightly better than $P_2$. Notice that the players starting in even-numbered positions receive a ``bump'' in their probability of winning with this bump becoming less pronounced as the game goes on. Ignoring this even-numbered bump, the further back in line a player starts, the better that player's probability of winning. An explanation of this phenomenon is offered in Section~\ref{subsec:second}.

\subsection{First is the Worst}
\label{subsec:first}
It turns out that the playground wisdom of ``first is the worst'' is at least half right. By Corollary~\ref{cor:2player}, in the two-player game, the probability that $P_1$ wins always lies in the interval $[\frac{1}{3}, \frac{1}{2})$. In particular, notice that $P_1$ never has an equal probability of winning as $P_2$. For the $n$-player game where $n < 60$, we searched for any case where the first player does not have the minimum probability of winning where $q$ ranges from 0.01 to 1.00 and $p$ ranges from 0.00 to 0.99, in increments of 0.01, skipping the case where $q = 1.00$ and $p = 0.00$ because this game ends in an infinite loop of players missing their first (long) shot and immediately making their second (short) shot where no players are eliminated. In the graphs that follow, notice that $P_1$ often has a much lower probability of winning than any other player. Thus, not only is first the worst, it is the worst by a rather large margin.

It is not hard to rationalize why $P_1$ has a disadvantage. Most players enjoy a sort of immunity when they shoot their first shot. That is, usually players take their long shot while the player in front of them is either rebounding the ball or about to shoot a short shot. As such, the player shooting the long shot cannot be eliminated because the player behind them cannot make a basket before them if they do not yet have the ball. On the other hand, it is entirely possible for $P_1$ to be eliminated immediately after missing their first shot, leaving them to both feel the sting of being the first player eliminated and also having to wait until almost an entire game of Knockout is played before they get to play again.

This property of Knockout can be thought of as a feature, rather than a bug. The second author often remembers  in his childhood when playing Knockout, new games would start with the order being determined by the reverse order of elimination. That is, the winner would be ``rewarded'' for their win by being forced to start in the first position, with second place as the second position. This allowed games to often have different winners since the last player that won (who may be best shooter) starts in the worst position possible. The downside of this strategy is that the first person to be eliminated (who is often a weaker player) starts as the last person in line. If all players were equally skilled, this would be an advantage, but on the playground where players have varying levels of abilities, this often leads the weakest player having the best or second to best player in a position to eliminate them.

\subsection{Second is... better}
\label{subsec:second}
Players that start in an even numbered position tend to receive a bump in their probability of winning, but that bump is most pronounced for the second player, less pronounced for the fourth player and so on, until it appears to go away completely. This can be seen in Figure~\ref{fig10} with 10 players or in Figure~\ref{fig100} with 100 players. Decreasing the long shot probability appears to increase the bump. It is particularly noteworthy how large the second player bump is in Figure~\ref{fig100} with 100 players. Notice that while the fourth, sixth, and eighth players also enjoy an increased probability of winning the game, the most pronounced increase is enjoyed by the second player, with this bump quickly tapering off until it can almost not be seen after the 10th player in line. 

\begin{figure}
\begin{center}
\includegraphics[width=\textwidth]{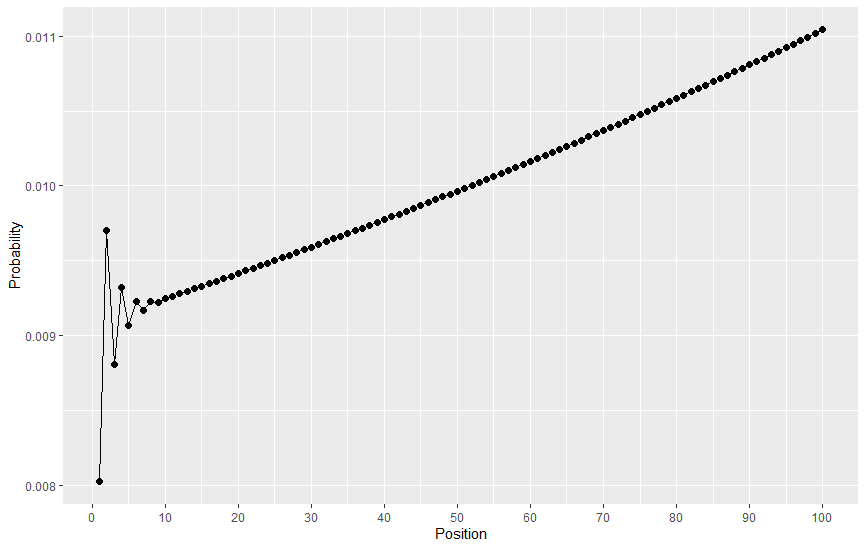}
\end{center}
\caption{Probability of winning by position when $n=100$, $p=0.4$ and $q = 0.9$}
\label{fig100}
\end{figure}

If it is too difficult to see the probabilities with 100 players, it may be instructive to examine the similar behaviors in the 10-player game in Figure~\ref{fig10}. Notice once again while all even numbered players receive an increased probability of winning compared to the odd numbered players, this advantage decreases as the players are later in line. We can reason why this is true. The ``immunity'' that the second player receives enables that player to likely not be eliminated after their first chances to make a basket. This allows them to cycle to the end of the line and enjoy a similar advantage that the last player in line enjoys. However, there is a (relatively smaller) chance that the second player is eliminated right away so often their chances of winning is not as great as the last player in line.

\begin{figure}
\begin{center}
\includegraphics[width=\textwidth]{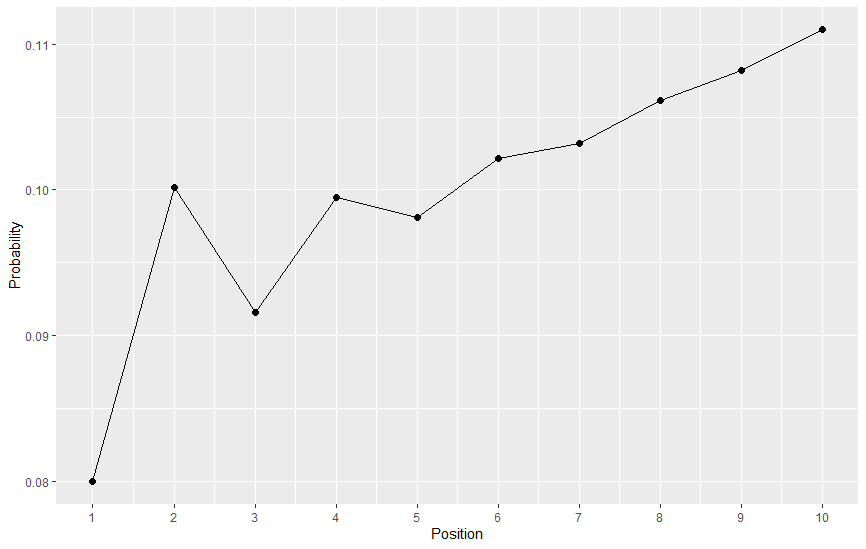}
\end{center}
\caption{Probability of winning by position when $n=10$, $p=0.4$ and $q = 0.9$}
\label{fig10}
\end{figure}

As mentioned above, decreasing the probability of making a long shot appears to increase the even-numbered position bump. Figure~\ref{fig:multiple} is a graph that shows this behavior with 10 players who shoot a short shot probability of $q=90\%$ and the long shot varies. The circles represent a long shot probability of $p=20\%$, the squares represent a long shot probability of $p=50\%$, and the triangles represent a long shot probability of $p=80\%$. Notice that the oscillation between even and odd positioned players is most pronounced for the game where $p=20\%$ and less pronounced for $p=80\%$, with the game where $p=50\%$ falls somewhere in the middle. The precise probabilities are represented in Table \ref{table:multiple}.

Instead of viewing the even-numbered players as getting a bump in winning probability, another way of viewing this phenomenon is that the odd-numbered players are at a disadvantage. It is not difficult to see that this is certainly the case for $P_1$\ Indeed, the first player may be immediately eliminated if they miss their first shot (if the $P_2$ also makes their first shot). Looking back at Figure \ref{fig100}, notice that $P_1$ has a considerably lower probability of winning, followed by the third, fifth, and seventh players in line. Again as we move further back in line, this disadvantage dissipates. Looking ahead at Figure \ref{fig:multiple}, it may be at first confusing to see that the probability of the first player winning is highest for $p=0.8$ and lowest for $p=0.5$ with the probability $P_1$ wins for $p=0.2$ falling somewhere in the middle. Investigating this further, it appears that this is closely related to the probability that$P_1$ is eliminated in their very first go-around shooting. See Section \ref{subsec:unfairgames} for more details.

\begin{figure}
\begin{center}
\includegraphics[width=\textwidth]{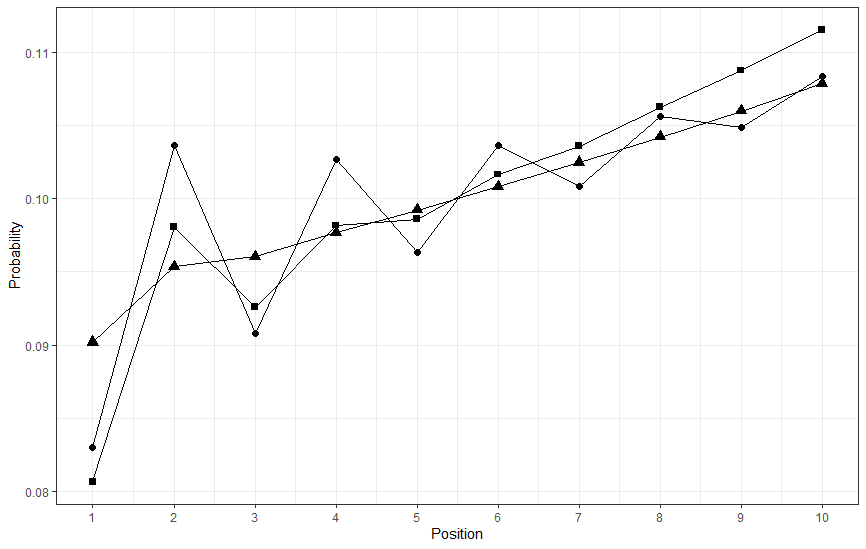}
\end{center}
\caption{Probability of winning by position when $n=10$, $q = 0.9$ and $p$ varies between $0.2, 0.5$, and $0.8$.}
\label{fig:multiple}
\end{figure}

\begin{table}
\begin{center}
\begin{tabular}{|c|c|c|c|c|c|c|c|c|c|c|c|}
\hline
Position: & 1 & 2 & 3 & 4 & 5 & 6 & 7 & 8 & 9 & 10 \\
\hline
$p = 0.2$: & 0.0830 &0.1037 & 0.0908 & 0.1027 &0.0964 & 0.1037 & 0.1009 &0.1057 &0.1049& 0.1084 \\
\hline
$p = 0.5$: & 0.0806 & 0.0981 &0.0926 &0.0982 &0.0986 &0.1017& 0.1036 & 0.1063 & 0.1088 &0.1116 \\
\hline
$p = 0.8$: & 0.0902 & 0.0954 & 0.0961 & 0.0977& 0.0992 &0.1008 &0.1025 &0.1042& 0.1060 &0.1079 \\
\hline
\end{tabular}
\end{center}
\caption{Probability of winning by position when $n=10$, $q=0.9$ and $p$ ranges}
\label{table:multiple}
\end{table}

\subsection{Last but not least}
\label{subsec:last}
In addition to not wanting to be first in line, it is better to be further back in line  as a general rule of thumb. In fact, a computer search indicated that as $n$ ranges from 2 to 60, and $q$ ranges from $0.01$ to $1.00$ while $p$ ranges from 0.00 to 0.99 in increments of 0.01, about 97\% of the time it is best to be the very last player in line. All of these cases where being last is not the optimal starting position occur when $n$ is odd and $p$ is relatively low. In each of these cases, the optimal starting position is second to last. 

\subsection{When does this really matter?}
\label{subsec:unfairgames}
The goal of this article is to discover the optimal starting position in a game of Knockout. It is also reasonable to consider when this advantage matters. That is, we can ask which shooting percentages result in the least fair game - the games where it most matters what the starting positions of the players are. One way to measure the fairness of a game of Knockout is to examine the spread of the probabilities of the most advantaged player with the least advantaged player by position. In using this metric as the definition of a ``least fair'' game of Knockout, the data indicates that the least fair games occur when $p \approx 0.5$. The authors believe the reason behind this value of $p$ maximizing the difference in probabilities of winning is due to this value maximizing the probability that $P_1$ is eliminated in their first turn with the ball. Indeed the probability that $P_1$ is eliminated before passing the ball to $P_3$ is
$$(1-p)p + (1-p)^2(1-q)q\left(\dfrac{1}{1-(1-q)^2}\right) \approx (1-p)p$$
because the probability of making a short shot should be assumed to be rather close to one so $1-q$ is small. The quantity $(1-p)p$ is maximized when $p=0.5$. So the first player's probability of being eliminated right away is highest when $p=0.5$. As mentioned in Section~\ref{subsec:first}, $P_1$ the lowest probability of winning. Increasing the probability that this player is eliminated right away decreases the probability this player wins which, in turn, increases the maximum spread of of all players' win probability. Therefore, the starting position in a game of Knockout matters most when all players shoot about 50\% from the long shot.

\subsection{Tips to Win at Knockout}
As a summary, if one finds themself involved in a game of Knockout - even if this game is not made up of clones of oneself - the authors recommend the following tactics for choosing a starting position to optimize the player's chance of winning.
\begin{enumerate}
	\item Do not be the first player in line. According to our simulations, the first player in line always has the worst probability of winning. This is the most important point. It is also more fun to stay in the game longer and it is not unusual for the first player to line to also be the first player that is eliminated.
	\item Be as far back in line as possible. As discussed in Section \ref{subsec:last} and seen in the graphs of the probabilities of winning, it is usually better to be as far back in line as possible.
	\item Choose an even-numbered position in line. This tip matters less if the player is successful in the tip above, being as far back in line as possible. Although many players (and playground wisdom) believe ``second is the best,'' the data shows that this is almost never the case. If given the choice between being second and fourth, it is sometimes better to be in the second position, especially if most players shoot a low percentage from long range, but for the sake of a general rule of thumb it is usually better to be later in line in a second position. Given the choice of between 4th and 5th, though, things get a little trickier. Now the question becomes which is more advantageous - the even position or the further back in line? This is where the more detailed discussion in Section \ref{subsec:second} becomes relevant. Since the even numbered bump gets less pronounced the further back in line players move, the even-odd position in line matters less and less. For a simple rule of thumb, the authors recommend choosing the even position if the positions are in the first quarter of the line and the choice of the two positions are concurrent. If the two positions differ by more than one or the positions are not in the first quarter of the line, it is usually a good bet to just be as far back in line as possible, because the even numbered position bump will have almost phased out by this point.
	\item Avoid being directly in front of or behind a good shooter. While the model does not account for differences in individual players' shooting abilities, following the logic of the trends outlined in Section \ref{sec:numeric} provides this tip. Trying to avoid being directly in front of a good shooter is rather obvious. Since the player behind you is the one that can eliminate you, you do not want them to have a high probability of making shots. Perhaps less obvious is that it is not a good idea to also be directly behind a good shooter. This may seem counter-intuitive because being behind the good shooter means that the last person the good shooter can eliminate is you. However, each time the player in front of you eliminates a player, you are in the set-up of a new round of Knockout (with one less player) where you are in the first position. As discussed in Section \ref{subsec:first}, first is the worst. Since good shooters tend to eliminate more players, this means that being directly behind a good shooter puts you in the disadvantageous first position more often.
	\item Finally, if all else fails, one should be a better shooter than the other players.
\end{enumerate}

%%%%%%%%%%%%%%%%%%%%%%%%%%%%%%%%%%%%%%
\section{Unequally Skilled Players and Questions}
\label{sec:unequal}
%%%%%%%%%%%%%%%%%%%%%%%%%%%%%%%%%%%%%%
The majority of this paper limits the scope to studying the game of Knockout where all players are equally skilled. The reason for this assumption is to reduce the number of parameters. To account for unequally skilled players requires $2n+1$ parameters: $n$ and each player's individual long shot and short shot percentage. With the equally-skilled assumption, there are only three variables to consider: $n$, $p$, and $q$.

This does not match the real world where people player this game with vastly different skill levels. The same analysis as above can be done where each player $P_k$ has a different long shot percentage $p_k$ and short shot percentage $q_k$. The case of two players with different skill levels is very similar to Equation~\ref{eqn:twoplayer} and a computer can calculate the probability of each player winning. Unfortunately the denominator is a polynomial of degree 6 containing 25 terms which makes it unable to fit on this page. There are still interesting questions that one can ask and investigate when the players are unequally skilled.

%\begin{center}
%\resizebox{1\textwidth}{!}{$\dfrac{\left(-p_A^{2}p_B\,q_B^{2}+2\,p_A\,p_B\,q_A\,q_B^{2}-p_B\,q_A^{2}q_B^{2}-2\,p_A\,p_B\,q_A\,q_B+2\,p_B\,q_A^{2}q_B+p_A^{2}q_B^{2}-2\,p_A\,q_A\,q_B^{2}+q_A^{2}q_B^{2}-p_B\,q_A^{2}+2\,p_A\,q_A\,q_B-2\,q_A^{2}q_B+q_A^{2}\right)}{\left(p_A^{2}p_B^{2}q_A\,q_B-p_A\,p_B^{2}q_A^{2}q_B-p_A^{2}p_B\,q_A\,q_B^{2}+p_A\,p_B\,q_A^{2}q_B^{2}-p_A^{2}p_B\,q_A\,q_B-p_A\,p_B^{2}q_A\,q_B+2\,p_A\,p_B\,q_A^{2}q_B+p_B^{2}q_A^{2}q_B+p_A^{2}q_A\,q_B^{2}+2\,p_A\,p_B\,q_A\,q_B^{2}-2\,p_A\,q_A^{2}q_B^{2}-2\,p_B\,q_A^{2}q_B^{2}-p_A\,p_B\,q_A^{2}-2\,p_A\,p_B\,q_A\,q_B+p_A\,q_A^{2}q_B-p_A\,p_B\,q_B^{2}+p_B\,q_A\,q_B^{2}+3\,q_A^{2}q_B^{2}+p_A\,q_A\,q_B+p_B\,q_A\,q_B-3\,q_A^{2}q_B-3\,q_A\,q_B^{2}+q_A^{2}+q_A\,q_B+q_B^{2}\right)}$}
%\end{center}

\begin{question}
	If $n-1$ equally skilled players play against an expert shooter such as an NBA player, what is the best starting position relative to the NBA player? Does the relative position change depending on the NBA player's starting position?
\end{question}

For example, suppose a group of $n-1$ mathematicians played against Caitlin Clark or Stephen Curry. Since Clark and Curry have a massive shooting advantage over the mathematicians, where should the mathematicians force these superstars to stand in line at the beginning of the game in order to maximize the chances that a mathematician wins? The authors believe the answer is almost certainly forcing the superstar to go first. It would also be interesting to ask how many mathematicians must be added so that the superstar's probability of winning is less than 50\%.

Another question to explore inverts the question involved in this paper
\begin{question}
	If $P_1$ is first in line and playing Knockout with $n-1$ other players who are equally skilled, how much better does $P_1$ have to be in order to have the best probability of winning?
\end{question}

Answering the questions above and more questions regarding unequally skilled players could make for an interesting project for a motivated undergraduate student.

%%%%%%%%%%%%%%%%%%%%%%%%%%%%%%%%%%%%%%
%%%%%%%%%%%%%%%%%%%%%%%%%%%%%%%%%%%%%%
\raggedbottom
% \addcontentsline{toc}{section}{\numberline{}}
% \bibliographystyle{amsalpha}\bibliography{bibliography}
\def\cprime{$'$} \def\cprime{$'$}
\providecommand{\MR}{\relax\ifhmode\unskip\space\fi MR }
\providecommand{\MRhref}[2]{%
  \href{http://www.ams.org/mathscinet-getitem?mr=#1}{#2}
}
\providecommand{\href}[2]{#2}
%%%%%%%%%%%%%%%%%%%%%%%%%%%%%%%%%%%%%%

%%%%%%%%%%%%%%%%%%%%%%%%%%%%%%%%%%%%%%
\end{document}